\documentclass[12pt,reqno]{amsart}

\usepackage{amsthm}
\usepackage{amssymb}
\usepackage{latexsym}
\usepackage{multicol}
\usepackage{verbatim,enumerate}
\usepackage{hyperref}
\usepackage{amsmath, amscd}
\usepackage{mathrsfs}
\usepackage[all,cmtip]{xy}
\usepackage{soul}
\usepackage{palatino}
\advance\textwidth by 1.2in \advance\oddsidemargin by -.6in \advance\evensidemargin by -.6in
\usepackage{tikz}
\usetikzlibrary{decorations.markings}
\tikzstyle{vertex}=[ circle, fill, draw, inner sep=0pt, minimum size=4pt,]
\tikzstyle{edge}= [thick]

\newtheorem{theorem}{Theorem}
\newtheorem{proposition}{Proposition}
\newtheorem{lemma}{Lemma}
\newtheorem{corollary}{Corollary}
\newtheorem{remark}{Remark}

\theoremstyle{definition}
\newtheorem{definition}{Definition}
\newtheorem{example}{Example}
\DeclareMathOperator{\supp}{supp}
\newcommand{\cC}{\mathcal{C}}
\newcommand{\cP}{\mathcal{P}}

\newcommand{\kk}{\mathfrak{k}}
\newcommand{\p}{p}
\title[Unique factorization of tensor products]{Unique Factorization For Tensor Products of Parabolic Verma Modules}
\author[]{K.N. Raghavan}
\address{The Institute of Mathematical Sciences, A CI of Homi Bhabha National Institute, Chennai 600113, India}
\email{knr@imsc.res.in}
\author[]{V. Sathish Kumar}
\address{The Institute of Mathematical Sciences, A CI of Homi Bhabha National Institute, Chennai 600113, India}
\email{vsathish@imsc.res.in}
\author[]{R. Venkatesh}
\address{Department of Mathematics, Indian Institute of Science, Bangalore 560012}
\email{rvenkat@iisc.ac.in}
\author[]{Sankaran Viswanath}
\address{The Institute of Mathematical Sciences, A CI of Homi Bhabha National Institute, Chennai 600113, India}
\email{svis@imsc.res.in}
\date{}
\thanks{The first, second and fourth authors acknowledge partial funding from a DAE Apex Project grant to the Institute of Mathematical Sciences, Chennai.}
\begin{document}
\begin{abstract}
    Let $\mathfrak g$ be a symmetrizable Kac-Moody Lie algebra with Cartan subalgebra $\mathfrak h$. We prove a unique factorization property for tensor products of parabolic Verma modules. More generally, we prove unique factorization for products of characters of parabolic Verma modules when restricted to certain subalgebras of $\mathfrak h$. These include fixed point subalgebras of $\mathfrak h$ under subgroups of diagram automorphisms of $\mathfrak g$ and twisted graph automorphisms in the affine case.
\end{abstract}

\maketitle

\section{Introduction}

Investigating whether a given family of elements from a ring has the unique factorization property (UFP) is a well-studied problem in basic ring theory. 
In this paper, we are interested in studying a similar phenomenon for representations of Kac-Moody algebras. 
Let $\mathfrak g$ be a Kac-Moody algebra. Suppose $\{M_i\}_{i=1}^r$ and $\{N_j\}_{j=1}^s$ are  $\mathfrak g$--modules from a suitable category  satisfying 
\begin{equation}\label{tensorprod}
\bigotimes\limits_{i=1}^r M_i \cong \bigotimes\limits_{j=1}^s N_j
\end{equation} as $\mathfrak g$--modules, then we have the following natural questions:
\begin{enumerate}
    \item Are the number of factors on both sides of \eqref{tensorprod} equal, i.e., is $r=s$ ?
    \item If they are equal, can we compare the highest weights of $M_i$'s and $N_j$'s if they are highest-weight modules?
    \item Can we prove the unique factorization property for their characters (when defined and viewed as elements in the character ring)?
        \item Further, what more can one say about the individual modules $M_i$ and $N_j$? For example, are they isomorphic up to a permutation of factors ?

\end{enumerate}
In the literature, unique factorization theorems study these questions. For example, C. S. Rajan proved a unique factorization property for tensor products of finite-dimensional simple modules of a finite-dimensional simple Lie algebra $\mathfrak g$ in \cite{rajan}. Later in \cite{VV} and \cite{reif-venkatesh}, the authors extended Rajan's result suitably beyond the realm of finite dimensional simple Lie algebras. 

\medskip
All these papers \cite{rajan, VV, reif-venkatesh} study only the unique factorization property of tensor products of simple modules in some suitable categories. In this paper, we will consider some families of typically reducible modules and study their unique factorization properties. More precisely, we consider the following two families of modules of $\mathfrak g = \mathfrak g(A)$  symmetrizable Kac-Moody algebra (where 
$A$ is a $n\times n$ symmetrizable generalized Cartan matrix):
\begin{enumerate}
    \item Parabolic Verma modules of $\mathfrak g$: these are highest-weight modules subsuming the class of simple integrable highest-weight modules (but are typically neither simple nor integrable).
    \item Restrictions of parabolic Verma modules to suitable subalgebras of $\mathfrak g$, for example, to fixed point subalgebras of Dynkin diagram automorphisms.
\end{enumerate}
Let $\mathfrak h$ be a fixed Cartan subalgebra of $\mathfrak g$ and $\{\alpha_i^\vee : 1\le i\le n \}$ be co-roots corresponding to the simple roots with respect to $\mathfrak h.$
The parabolic Verma modules of $\mathfrak g$ are indexed by 
$(\lambda, I)$, where $\lambda\in \mathfrak h^*$ and $I\subseteq \{1, \ldots, n\}$ and we denote by $M(\lambda, I)$ the parabolic Verma module corresponding to the tuple $(\lambda, I).$
We call a subset $I\subseteq \{1, \ldots, n\}$  \textit{connected} if when $I$ is thought of as a subset of the nodes of the Dynkin diagram associated with $A$, the subgraph induced by $I$ is connected.
Here is our main theorem for parabolic Verma modules.
\begin{theorem}\label{thm_intro_ufpvm}
   Let $A$ be a $n \times n$ symmetrizable generalized Cartan matrix. Let $\mathfrak g = \mathfrak g (A)$ be the Kac-Moody Lie algebra associated with $A$ and let $\mathfrak h$ be the Cartan subalgebra. Suppose that
    \begin{equation}\label{tensorproductpvm}
        \bigotimes\limits_{k=1}^r M(\lambda_k, I_k) \cong \bigotimes\limits_{k=1}^r M(\mu_k, J_k)
    \end{equation}
    where for each $1\le k\le n$ we have
    \begin{enumerate}
        \item $\lambda_k, \mu_k \in \mathfrak h ^*$ and $I_k, J_k$ are connected subsets of $\{1, \ldots, n\}$,
        \smallskip\item $\lambda_k (\alpha_i ^\vee),\mu_k(\alpha_j^\vee)$ are positive integers for all $i\in I_k$ and $j\in J_k$.
    \end{enumerate}
    Then, $\sum_{k = 1 } ^r \lambda_k = \sum_{k=1} ^r \mu_k$ and there exists a permutation $\sigma \in \mathfrak S _r$ such that 
        $$\text{ $I_k = J_{\sigma k}$ and
         $\lambda_k(\alpha_i ^\vee) = \mu_{\sigma k}(\alpha_i ^\vee)$ for all $i\in I_k$.}$$
\end{theorem}

\medskip
\noindent
Since $M(\lambda, I)$ are simple integrable highest weight modules when $\lambda$ is dominant and $I=\{1, \ldots, n\}$, this may be viewed as an extension of the results of \cite{rajan, VV}.

\medskip
Next, let $\Gamma$ be a subgroup of Dynkin diagram automorphisms of $\mathfrak g$. Consider the fixed point subalgebra $\mathfrak g^\Gamma$ of $\mathfrak g$ with respect to $\Gamma$. Then we can restrict the modules in Equation~\eqref{tensorproductpvm} to $\mathfrak g^\Gamma$ and ask whether the unique factorization property holds for these $\mathfrak g^\Gamma$ modules. Under some natural conditions, we answer this question affirmatively. More precisely, we prove the following:
\begin{theorem}\label{mainthm2intro}
    Let $A$ be a $n \times n$ symmetrizable generalized Cartan matrix of finite, affine, or hyperbolic type. Let $\mathfrak g = \mathfrak g (A)$ be the associated Kac-Moody Lie algebra, with Cartan subalgebra $\mathfrak h$. Let $\Gamma$ be a group of diagram automorphisms of $\mathfrak g$ and let $\mathfrak g ^{\Gamma}$ be the fixed point subalgebra. Suppose that
    \begin{equation*}
        \bigotimes _{k=1} ^r Res_{\mathfrak g ^{\Gamma}} M(\lambda_k, I_k) = \bigotimes _{k=1} ^r Res_{\mathfrak g ^{\Gamma}} M(\mu_{k}, J_{k})
    \end{equation*}
    where for each $1\le k \le r$, we have 
    \begin{enumerate}
        \item $\lambda_k, \mu_k\in (\mathfrak h ^*)^\Gamma = \{\nu \in \mathfrak h^* : \nu(\omega(h)) = \nu(h) \; \text{for all} \; \omega\in \Gamma, \, h \in \mathfrak h\}$ and $I_k, J_k$ are connected subsets of $\{1,\ldots, n\}$,
        \smallskip\item for each $i \in I_k$ and $j\in J_k$ we have $\lambda_k (\alpha_i ^\vee)$, $\mu_k (\alpha_j ^\vee)$ are positive integers, and
        \smallskip\item $I_k$ and $J_k$ are $\Gamma$-stable, i.e., are unions of $\Gamma$-orbits. 
    \end{enumerate}
    Then, $\sum_{k = 1 } ^r \lambda_k = \sum_{k=1} ^r \mu_k$ and there exists a permutation $\sigma \in \mathfrak S _r$ such that 
$$\text{ $I_k = J_{\sigma k}$ and
         $\lambda_k(\alpha_i ^\vee) = \mu_{\sigma k}(\alpha_i ^\vee)$ for all $i\in I_k$.}$$
\end{theorem}
\noindent
We in fact prove stronger versions of Theorems ~ \ref{thm_intro_ufpvm}, \ref{mainthm2intro}, see $\S 4 \text{ and } \S 5$ for more precise statements. 
It is to be noted that in \cite{sas} a similar theorem is proved in the setting of simple modules for finite-dimensional simple Lie algebras with a completely different set of hypotheses. 
It is easy to see that the converse of Theorems~\ref{thm_intro_ufpvm}, \ref{mainthm2intro} hold at the level of characters. Further,  if we assume that complete reducibility holds for the tensor products, then we can also prove the converse of the Theorems~\ref{thm_intro_ufpvm}, \ref{mainthm2intro}. For example, the converse is true (see \cite[Page No. 180, Corollary 10.7]{Kac90}) when
\begin{enumerate}
    \item   all $\lambda_k$'s and $\mu_k$'s are integral dominant weights and $I_k=J_k=\{1,\ldots, n\}$ (for all $1\le k\le n$) in Theorem~\ref{thm_intro_ufpvm}  and
    \item the fixed point subalgebra $\mathfrak{g}^\Gamma$ is again Kac-Moody type and all $\lambda_k$'s and $\mu_k$'s are integral dominant weights and $I_k=J_k=\{1,\ldots, n\}$ (for all $1\le k\le n$)
 in Theorem~\ref{mainthm2intro}. 
\end{enumerate}

\medskip
The paper is organized as follows: In Section~\ref{section2}, we set up the notation and preliminaries.
In Section~\ref{section3}, we prove some key technical results that will be needed to prove our main theorems. 
In Section~\ref{section4}, we prove our main theorem for parabolic Verma modules. In Section~\ref{section5}, we consider and prove unique factorization properties for the characters of restricted parabolic Verma modules. We apply this in Section~\ref{section6} and prove unique factorization properties for parabolic Verma modules when they are restricted to $\mathfrak g^\Gamma$, where $\mathfrak g$ is general and $\Gamma$ is a
subgroup of Dynkin diagram automorphisms or $\mathfrak g$ is of affine type and $\Gamma = \langle \tau \rangle$ for some $\tau$ twisted graph automorphism.

\section{Preliminaries}\label{section2}
All vector spaces are assumed to be defined over complex numbers $\mathbb{C}$ throughout the article. For a Lie algebra $\mathfrak g$, we denote by $U(\mathfrak{g})$  the universal enveloping algebra of $\mathfrak g.$ 
For a vector space $V$ over $\mathbb{C}$, we denote by $V^*$ its dual space.
\subsection{Structure Theory of Symmetrizable Kac-Moody algebras}
In this subsection, we fix some notation and review the structure theory of Kac-Moody algebras, closely following \cite{Kac90}.
Let $n$ be a positive integer and $A=(a_{ij})_{n\times n}$ a \emph{generalized Cartan matrix} (GCM). That is, 
\begin{enumerate}
    \item $a_{ii}=2$, for all $1\le i\le n$,
    \item $a_{ij}$ is a non-positive integer for all $1\le i\neq j\le n$, and
    \item $a_{ij}=0$ if and only if $a_{ji}=0$ for all $1\le i, j \le n$.
\end{enumerate}
We say $A$ is \emph{symmetrizable} if there exists a diagonal matrix $D=\mathrm{diag}(d_1, \ldots, d_n)$, with $d_i$'s being positive real numbers, such that $DA$ is symmetric. Let us denote $S=\{1,\ldots, n\}$.

\medskip
Let $A$ be a symmetrizable GCM and let $\mathfrak{g}:=\mathfrak{g}(A)$ be the Kac-Moody Lie algebra associated with $A$ and $\mathfrak{h}$ be a fixed Cartan subalgebra of $\mathfrak{g}$.
The Cartan subalgebra $\mathfrak{h}$ acts semisimply on $\mathfrak{g}$ via the adjoint action. Denoting by $\Delta$ the set of roots of $(\mathfrak{g}, \mathfrak{h})$, the corresponding root space decomposition is
$$\mathfrak{g} = \mathfrak{h} \oplus \bigoplus_{\alpha\in \Delta} \mathfrak{g}_\alpha,$$
where $\mathfrak{g}_\alpha:=\{x\in \mathfrak{g} : [h, x]=\alpha(h)x\;\forall h\in \mathfrak{h}\}$ for $\alpha\in \Delta$. 

\medskip
Let the simple system (i.e., simple roots) of $\Delta$ coming from the realization of $A$ be
$\Pi := \{\alpha_1, \cdots, \alpha_n\}$. We denote by $\Delta_+$ the set of positive roots of $\Delta$ with respect to $\Pi$. The root lattice $Q$ is the set of all integer linear combinations of elements of $\Pi$. Any $\alpha \in Q$ can be written uniquely as $\sum_{i=1}^{n} t_i \alpha_i$ where $t_i$ are integers. The set of all $\alpha \in Q$ for which all the $t_i$ are non-negative integers is denoted $Q^+$. The support of $\alpha \in Q$ denoted by $\supp(\alpha)$ is the set of all $k \in S$ for which $t_k \neq 0$. 
For $\alpha\in \Pi$, let $\alpha ^\vee \in \mathfrak{h}$ denote the coroot corresponding to $\alpha.$ Let $$\mathfrak{h}\cup \{e_i, f_i : i\in S\}$$ be the Chevalley generators of $\mathfrak{g}$.  Note that the derived subalgebra $[\mathfrak{g}, \mathfrak{g}]$ is generated by $\{e_i, f_i : i\in S\}$ and $[\mathfrak g, \mathfrak g] \cap \mathfrak h = \text{span}\{\alpha_1 ^\vee, \ldots, \alpha_n ^\vee\}$.
The Weyl group of $\mathfrak{g}$ is the subgroup of $GL(\mathfrak{h}^*)$ generated by the reflections $\{s_i : i\in S\}$, where $s_i:\mathfrak{h}^*\to \mathfrak{h}^*$ is defined by $$s_i(\nu)=\nu-\nu(\alpha_i^\vee)\alpha_i \quad \quad \forall \nu \in \mathfrak{h}^*$$
The parabolic subgroup $W_I$ corresponding to $I \subseteq S$ is the subgroup of $W$ generated by $\{s_i : i\in I\}$. 
It is a fact that $W_I$ is a Coxeter group with Coxeter generators $\{s_i: i \in I\}$.  
Given $w\in W_I$,  
the length of $w$ is 
$\ell(w):=\mathrm{min} \{k : w=s_{i_1}\cdots s_{i_k}\}$; an expression $s_{i_1} \cdots s_{i_k}$ for $w$ is said to be reduced if $k = \ell(w)$.
Fix a reduced expression $s_{i_1}\cdots s_{i_k}$ of $w\in W_I$. Define the support of $w$ by
$$I(w):=\{s_{i_1},\ldots, s_{i_k}\}.$$ 
It is a well-known fact (\emph{Tits theorem}) that $I(w)$ is independent of the choice of the chosen reduced expression.

\medskip
A subset $I\subseteq S$ is said to be \textit{connected} if the submatrix of the GCM indexed by $I$ is indecomposable (or equivalently, the subgraph of the Dynkin graph of $\mathfrak g$ induced by $I$ is connected). Note that any subset $I$ of $S$ can be written as a finite disjoint union of connected subsets, called connected components, of $I$ and this decomposition is unique up to a permutation of the connected components.

\subsection{Parabolic Verma Modules}\label{parabolicVerma}
For $\lambda\in \mathfrak{h}^*$, we denote by $M(\lambda)$ the Verma module associated to $\lambda$. The \emph{integrability} \label{integrability} of $\lambda$ is defined by
$$J_{\lambda} := \{i \in S\;|\; \lambda(\alpha_i ^{\vee}) \text{ is a non-negative integer} \}$$
For $\lambda \in \mathfrak{h}^*$ and $I \subseteq J_{\lambda},$ the \emph{Parabolic Verma Module} corresponding to $(\lambda, I)$ is defined by
    \begin{equation*}
        M(\lambda, I) := \frac{M(\lambda)}{\sum_{i \in I} U(\mathfrak{g})f_{\alpha_i} ^{\lambda(\alpha_i ^{\vee}) + 1} m_{\lambda}}
    \end{equation*}
 where $m_\lambda$ is the cyclic generator (or highest weight vector) of $M(\lambda)$. The parabolic Verma modules are modules in category $\mathcal{O}$ (see \cite{ apoorva-dhillon, BGGO} for more details). When $\lambda$ is dominant and integral (i.e., $J_{\lambda} = S$) we see that $M(\lambda, S)=V(\lambda)$, the unique simple $\mathfrak g$ module with highest weight $\lambda$. When $I = \emptyset$ we have
 $M(\lambda, I)=M(\lambda)$ for all $\lambda\in \mathfrak h^*$. So, the parabolic Verma modules interpolate between the Verma modules and simple modules in category $\mathcal{O}$.

\subsection{Characters and their restrictions}
Let $\bar{\mathfrak{g}}$ be a Lie subalgebra of $ \mathfrak{g}$ 
and $\bar{\mathfrak{h}} := \bar{\mathfrak{g}} \cap \mathfrak{h}$. Suppose $V$ is a $\mathfrak{h}$--weight module and 
$$V=\bigoplus\limits_{\nu \in \mathfrak{h}^*} V_\nu$$
is the $\mathfrak{h}$--weight space decomposition of $V$. Note that by definition, we have $\mathrm{dim} V_{\nu}<\infty$ for all $\nu\in \mathfrak{h}^*.$
The $\mathfrak{h}$--character of $V$ is defined by
$$\mathrm{ch}_{\mathfrak{h}}(V) = \sum _{\nu \in \mathfrak{h}^*} \mathrm{dim} (V_{\nu}) e^{\nu}.$$
Let $p:\mathfrak{h}^*\to \bar{\mathfrak{h}}^*$ be the restriction map. The restriction of the character of $V$ to $\bar{\mathfrak h}$ is defined as
 $$\mathrm{ch}_{\bar{\mathfrak{h}}}(V) = 
 \sum _{\nu' \in (\bar{\mathfrak{h}})^*} \left(\sum_{\{\nu \in \mathfrak{h}^*\, : \, p(\nu) = \nu'\}} \mathrm{dim} V_{\nu}\right) e^{\nu'}.$$
Note that the inner sum $\sum_{\{\nu \in \mathfrak{h}^*\, : \, p(\nu) = \nu'\}} \mathrm{dim} V_{\nu}$ need not be finite always. 
We use this definition whenever it makes sense, i.e., whenever this inner sum is finite.
This inner sum will be finite and the restricted character $\mathrm{ch}_{\bar{\mathfrak{h}}}(V)$ will be well-defined in all our examples.

\subsection{The Weyl-Kac character formula}
For $\lambda\in P^+$, the Weyl-Kac character formula gives
$$\mathrm{ch}_\mathfrak{h}\, V(\lambda)=\frac{\sum\limits_{w\in W}(-1)^{\ell(w)}e^{w(\lambda+\rho)-\rho}}{
\prod_{\alpha \in \Delta^+} (1-e^{-\alpha})^{\mathrm{dim}\, \mathfrak{g}_{\alpha}}}$$
\smallskip where $\rho \in \mathfrak h ^*$ is some fixed functional that satisfies $\rho(\alpha_i ^\vee) = 1$. A similar formula is known for parabolic Verma modules. The following proposition can be found in \cite[Proposition 7.10]{apoorva-dhillon}.
    Let $\lambda \in \mathfrak{h}^*$ and $I \subseteq J_{\lambda}$. Then
    \begin{equation}\label{parabolicverma}
    \mathrm{ch}_\mathfrak{h}\ M(\lambda, I) =  \frac{\sum_{w \in W_I} (-1)^{\ell(w)}e^{w(\lambda+\rho) - \rho}}{\prod_{\alpha \in \Delta^+} (1-e^{-\alpha})^{\mathrm{dim}\, \mathfrak{g}_{\alpha}}}.
    \end{equation}
\noindent
Let $\lambda\in \mathfrak{h}^*$ and $I \subseteq J_{\lambda}$.
We define the \emph{normalised Weyl numerator} corresponding to the tuple $(\lambda, I)$ by 
\begin{equation*}
    U(\lambda, I) := \; e^{-\lambda}\sum_{w \in W_I} (-1)^{\ell(w)}e^{w(\lambda+\rho) - \rho}
\end{equation*}
In this notation 
we can rewrite Equation~\eqref{parabolicverma} as
\begin{equation*}
    \mathrm{ch}_\mathfrak{h}\ M(\lambda, I) = e^{\lambda} \frac{U(\lambda, I)}{U(0,S)}.
\end{equation*}

\medskip
\section{Technical results}\label{section3}
In this section, we will prove some technical results that will be used later to prove our main theorems.
We freely use the notations that were developed in the previous section. Let us define $\cP:= \{(\lambda, I)| \lambda \in \mathfrak h ^* \text{ and } I \subseteq J_{\lambda}\}$. This is the indexing set for the Parabolic Verma modules of $\mathfrak g$.
\subsection{}
The following is elementary; see, for example,  \cite[\S4.1 (in particular, Lemma 2)]{VV}.
\begin{proposition}\label{old_props_from_VV}
    Let $(\lambda,I) \in \cP$ and $w \in W_I$. Then  
    \begin{enumerate}
        \item $\lambda+\rho - w(\lambda + \rho) \in Q^+,$
         \item $\supp(w(\lambda+\rho) - (\lambda + \rho)) = I(w)$
    \end{enumerate}
\end{proposition}
\noindent   Note that  $U(\lambda, I)$ can be viewed as a formal power series in the variables $\{x_i:=e^{-\alpha _i}: i \in I\}$ by Proposition~\ref{old_props_from_VV}.  
  For $(\lambda, I) \in \cP$, we let
    $$ L (\lambda, I) := -log(U(\lambda, I)) = \sum\limits_{\alpha \in Q} c^{\lambda, I} _{\alpha}\; e^{-\alpha}$$
    \noindent Here, the logarithm is applied to $U(\lambda, I)$ treating it as a formal power series whose constant term is $1$. Note that
             $$c^{\lambda, I} _{\alpha}=0 \,
            \text{if}\, \alpha\notin Q^+.$$
\noindent
    We need some additional notations which we collect here:
\begin{align*}
    &\cC := \{(\lambda, I) \in \cP| \; I \;\text{is connected and nonempty}\}\\
  &\beta(\lambda, I) := \sum\limits_{i \in I} (\lambda + \rho)(\alpha_i ^{\vee})\, \alpha_i \in Q^+ \text{ for } (\lambda, I) \in \cP.
\end{align*}

\subsection{} We need the following key results.
\begin{proposition} \label{prop_new} Let $(\lambda, I) \in \cP$ and let $I=I_1\dot\cup \cdots \dot\cup I_r$ be the decomposition of $I$ into connected components. 
Then 
    \begin{enumerate}
        \item $(\lambda, J) \in \cP$ for all $J \subseteq I$.
        \smallskip\item  $L(\lambda, I) = L(\mu, J) \iff U(\lambda, I) = U(\mu, J) \iff \beta(\lambda, I) = \beta(\mu, J) \iff I=J\; and \;  \lambda(\alpha_i ^{\vee}) = \mu(\alpha_i ^{\vee})$ for all $i \in I$.
        \smallskip\item $L(\lambda, I) = \sum\limits_{k=1}^r L(\lambda, I_k)$.
        \smallskip\item $c^{\lambda, I} _{\alpha} = c^{\lambda, \supp(\alpha)} _{\alpha}$ for all $\alpha \in Q$ and $I \supseteq \supp(\alpha)$.
        \smallskip\item $c^{\lambda, I} _{\alpha} \neq 0$ implies that $\supp (\alpha)$ is a connected subset of $I$.
        \smallskip\item $c^{\lambda, I} _{\alpha} \neq 0$ and $\supp(\alpha) = I$ implies $\alpha - \beta(\lambda, I) \in Q^+$.
    \end{enumerate}
\end{proposition}
\begin{proof} \null \quad
    \begin{enumerate} 
        \item Immediate from the definition.\smallskip
        \item Comparing monomials of the form $e^{-k\alpha_i},\; i \in S$ gives the equivalence. \smallskip
        \item Suppose $I = I_1 \cup I_2$ where $ \alpha_{i_1}( \alpha_{i_2} ^\vee ) = 0$ for all $i_1 \in I_1$ and $i_2 \in I_2$. Then $W_I = W_{I_1} \times W_{I_2}$ because the simple reflections $s_{\alpha_{i_1}}$ and $s_{\alpha_{i_2}}$ commute for $i_1 \in I_1$ and $i_2 \in I_2$. For $w \in W_I$, there exists unique $w_1 \in W_{I_1}$ and $w_2 \in W_{I_2}$ such that $w = w_1 w_2$. Therefore we have \begin{align*}
            w(\lambda+\rho) - (\lambda +\rho) &= w_1 w_2(\lambda+\rho) - (\lambda +\rho)\\
            &= w_1 w_2(\lambda+\rho) - w_1(\lambda +\rho)+w_1(\lambda+\rho) - (\lambda +\rho)\\
            &= w_1 [w_2(\lambda+\rho) - (\lambda +\rho)]+ [w_1(\lambda+\rho) - (\lambda +\rho)]\\
        \end{align*} 
        But $\supp(w_2(\lambda+\rho) - (\lambda +\rho)) \subseteq I_2$ and hence $w_1$ fixes it. Therefore we have $$w(\lambda+\rho) - (\lambda +\rho) = w_1(\lambda+\rho) - (\lambda +\rho)+w_2(\lambda+\rho) - (\lambda +\rho)$$
        Hence, 
        \begin{align*}
            U(\lambda,I) &= \sum_{w \in W_I} (-1)^{l(w)} e^{w(\lambda + \rho) - (\lambda + \rho)} \quad \quad
            = \sum_{(w_1,w_2) \in W_{I_1}\times W_{I_2}} (-1)^{l(w_1 w_2)} e^{w_1w_2(\lambda + \rho) - (\lambda + \rho)}\\
            &= \left(\sum_{w_1 \in W_{I_1}} (-1)^{l(w_1)} e^{w_1(\lambda + \rho) - (\lambda + \rho)}\right) \left(\sum_{w_2 \in W_{I_2}} (-1)^{l(w_2)}e^{w_2(\lambda + \rho) - (\lambda + \rho)}\right)\\
            &= U(\lambda, I_1) \cdot U(\lambda, I_2)
        \end{align*}
        Now taking $-log$ on both sides, we get $$L(\lambda, I) = L(\lambda, I_1) + L(\lambda, I_2)$$
        \item Fix $\alpha \in Q.$ Let  $\psi := U(\lambda, I) -1 -\zeta$ where $$\zeta := \sum _{w \in W_{\supp(\alpha)} \backslash \{e\}} (-1)^{l(w)} e^{w(\lambda+\rho) - (\lambda+\rho)}$$
\noindent
        Then $L(\lambda, I) = \sum_{k \geq 1} (-1)^k (\zeta + \psi)^k / k$. 
 But since the support of any monomial in $\psi$ is not a subset of $\supp(\alpha)$, it follows that $e^{-\alpha}$ has contributions only from $\sum_{k \geq 1} (-1)^k \zeta ^k / k = L(\lambda, \supp(\alpha))$. Therefore $c^{\lambda, I} _{\alpha} = c^{\lambda, \supp(\alpha)} _{\alpha}$. 
 \smallskip
        \item The fact that $\supp(\alpha) \subseteq I$ follows from part (2) of Proposition~\ref{old_props_from_VV}. 
        Now, suppose that $\supp(\alpha)$ is disconnected. Let $\supp(\alpha) = I_1 \dot \cup I_2$, where $I_1$ and $I_2$ are proper non-empty subsets of $\supp(\alpha)$ and  $\alpha_{i_1}(\alpha_{i_2}^\vee) = 0$ for all $i_1 \in I_1$ and $i_2 \in I_2$. By part (3) of this proposition, we have $L(\lambda, \supp(\alpha)) = L(\lambda, I_1) + L(\lambda, I_2)$. Therefore, $c^{\lambda, I} _{\alpha} = c^{\lambda, \supp(\alpha)} _{\alpha} = c^{\lambda, I_1} _{\alpha} + c^{\lambda, I_2} _{\alpha}$. But since $\supp(\alpha)$ is not a subset of $I_1$ or $I_2$ we have $c^{\lambda, I_1} _{\alpha} = c^{\lambda, I_2} _{\alpha} = 0$. 
        \smallskip
        \item The coefficient of $\alpha_i$ in $\lambda+\rho - w(\lambda + \rho)$ is either $0$ or greater than or equal to $(\lambda+\rho) (\alpha _i^\vee)$ for all $i \in S$. This fact is elementary to prove. For example, see \cite[Lemma 2(b)]{VV} where it is proved by induction on the length of $w$. Now this means that the variable $e^{-\alpha_i}$ has degree equal to $0$ or greater than $(\lambda+\rho)(\alpha_i ^\vee)$ in any monomial in $U(\lambda, I)$. But since any monomial in $L(\lambda, I)$ is a product of monomials from $U(\lambda, I)$, the proof follows. 
    \end{enumerate}
\end{proof}
The following proposition is very crucial and follows from Propositions~3 and 7 of \cite{VV} (also see part (4) of Proposition~\ref{prop_new}, see also Exercise 1.2 in \cite{Kac90}).
We give a sketch of the proof for reader's convenience.
\begin{proposition} \label{prop_imp} Let $(\lambda, I)\in \mathcal{P}$ as before.
    \begin{enumerate}
        \item $c^{\lambda, I} _{\beta(\lambda, I)}$ is independent of $\lambda$.
        i.e., $c^{\lambda, I} _{\beta(\lambda, I)} = c^{\mu, I} _{\beta(\mu, I)}$ if $(\mu, I) \in \mathcal{P}$ with $I \subseteq J_{\lambda} \cap J_{\mu}$.
        \item $c^{\lambda, I} _{\beta(\lambda, J)} >0$ for any non-empty, connected subset $J$ of $I$.
        \item In particular, if $(\lambda, I) \in \cC$ then $c^{\lambda, I} _{\beta(\lambda, I)} >0$.
    \end{enumerate}
\end{proposition}
\begin{proof} 
    \begin{enumerate}
        \item Consider the subgraph $\mathcal{G}_I$ of the Dynkin diagram of $\mathfrak g$ induced by $I$. We then have $$c^{\lambda, I} _{\beta(\lambda,I)} = (-1)^{|I|} \sum_{k \geq 1} \frac{(-1)^{k} \; |\mathcal{P}_k(\mathcal{G}_I)|}{k} $$ where $\mathcal{P}_k (\mathcal{G}_I)$ is the set of all $k$-tuples $(J_1, \ldots, J_k)$ of pairwise disjoint subsets of $I$ such that 
        \begin{enumerate}[(a)]
            \item $J_1 \; \dot\cup \; \cdots \; \dot\cup \; J_k = I$
            \item $J_i$ is totally disconnected (i.e., $\forall x, y \in J_i$ there is no edge between $x$ and $y$ in $\mathcal{G}_I$).
        \end{enumerate}
        The RHS is clearly independent of $\lambda$, see \cite[\S 4.3]{VV} for more details.
        \smallskip\item By (1), it is enough to consider $\lambda = 0$. By the Weyl denominator identity (for the parabolic subalgebra $\mathfrak{g} _I$) we have $$U(0, I) = \prod_{\alpha \in \Delta_+ (I)} (1-e^{-\alpha})^{\mathrm{mult}(\alpha)}$$ where, $\Delta_+ (I) = \mathbb{Z}\text{-span}\{\alpha_i : i \in I\} \cap \Delta _+$ or equivalently the set of positive roots for the parabolic subalgebra $\mathfrak g _I$. Applying $-log$ we get $$L(0,I) = \sum_{\alpha \in \Delta_+ (I)} \mathrm{mult}(\alpha) \sum_{k \geq 1} \frac{e^{-k\alpha}}{k}$$
        The coefficient of $\beta(0, I) = \sum_{i \in I} \alpha_i$ in $L(0,I)$ is then the multiplicity of $\sum_{i \in I} \alpha_i$. But since $I$ is connected we have this multiplicity to be a positive integer (see \cite[Lemma 1.6]{Kac90} and \cite[Proposition 4 \& 7]{VV}).
        \smallskip\item Follows immediately from (2).
    \end{enumerate}
\end{proof}
\section{Unique factorization for Parabolic Verma Modules}\label{section4}
In this section, we will prove the unique factorization of tensor products of parabolic Verma modules of $\mathfrak g.$
First, we analyze when a sum of finitely many $L(\lambda, I)$'s can be equal to another such sum. 
\subsection{}
The following relation $\succeq$ on $\cP$ will play an important role in this paper. Define $(\lambda, I) \succeq (\mu, J)$ if:

\begin{center}either $I \supsetneq J$
\quad \quad or \quad \quad $I = J$ and $\beta(\mu, J) - \beta(\lambda, I) \in Q^+$.
\end{center}

\noindent Observe that the latter part of this condition may be exchanged with 
\begin{center}
$I=J$ and $\lambda(\alpha_i ^\vee)\leq \mu(\alpha_i ^\vee) \; \; \text{for all} \; i \in I$.\end{center}
Note that this relation is reflexive, and transitive but not anti-symmetric. i.e., $(\lambda, I) \succeq (\mu, J)$ and $(\mu,J) \succeq (\lambda, I)$ does not imply that $(\lambda, I) = (\mu, J)$. 

\medskip
For $(\lambda, I)$ and $(\mu, J) \in \cP$ we write $(\lambda, I) \approx (\mu, J)$ if $I = J$ and  $\lambda(\alpha_i ^{\vee}) = \mu(\alpha_i ^{\vee})$ for all $i \in I$. This defines an equivalence relation on $\cP$.
Observe that $(\lambda, I) \approx (\mu, J)$ means that these pairs satisfy the equivalent conditions of part (2) of Proposition~\ref{prop_new}. The relation $\succeq$ now defines a partial order on $\cP/\!\approx$.

\medskip
\textbf{Caveat:} Even though $\succeq$ does not form a partial order on $\cP$ we will find it convenient nevertheless to talk about {\em maximal elements} in a subset of $\cP$. What we actually mean by saying $(\lambda_1, I_1)$ is maximal among   $\{(\lambda_1, I_1),$ $(\lambda_2, I_2),$ $\ldots,$ $(\lambda_k, I_k)\}$ is that:
\begin{center}
    if $(\lambda_j, I_j) \succeq (\lambda_1, I_1)$ for some $j$ then $(\lambda_1, I_1) \approx (\lambda_j, I_j)$
\end{center}
or equivalently, when thought of as elements of $\cP /\!\!\approx$, $(\lambda_1, I_1)$ is maximal among   $\{(\lambda_1, I_1),$ $(\lambda_2, I_2),$ $\ldots,$ $(\lambda_k, I_k)\}$.
\begin{lemma}\label{main_lemma}
    Let $r\geq 0$. Let $(\lambda_1, I_1), \ldots, (\lambda_r, I_r)$ be (not necessarily distinct) elements of $\cC$ and let $L:= \sum_{k=1} ^r L(\lambda_k, I_k)$. 
    \begin{enumerate}
     \item For $(\mu, J) \in \cP$, if the coefficient of $e^{-\beta(\mu, J)}$ in $L$ is non-zero, then $(\lambda_j, I_j) \succeq (\mu, J)$ for some $1 \leq j \leq r$.
     \smallskip\item If $(\lambda_j, I_j)$ is a maximal element with respect to $\succeq$ among $(\lambda_1, I_1), \ldots, (\lambda_r, I_r)$ then the coefficient of $e^{-\beta(\lambda_j, I_j)}$ in $L$ is positive. In particular, $L \neq 0$ if $r \neq 0$.
    \end{enumerate}
\end{lemma}
\begin{proof} \null \;
    \begin{enumerate} 
        \item By hypothesis we have $c^{\lambda_j, I_j} _{\beta(\mu, J)} \neq 0$ for some $1 \leq j \leq r$. By Proposition~\ref{prop_new} (5), we have $J \subseteq I_j$. If $J \subsetneq I_j$ then $(\lambda_j, I_j) \succeq (\mu, J)$. Suppose $J = I_j$. By part (4) and (6) of Proposition~\ref{prop_new} we have $\beta(\mu, J) - \beta(\lambda, J) \in Q^+$. Therefore $(\lambda_j, I_j) \succeq (\mu, J)$. 
        \smallskip\item By part (1) and maximality of $(\lambda_j, I_j)$, if the coefficient of $e^{-\beta(\lambda_j, I_j)}$ in $L(\lambda_k, I_k)$ is non-zero then $(\lambda_j, I_j) \approx (\lambda_k, I_k)$.  In such a case, $\beta(\lambda_j, I_j) = \beta(\lambda_k, I_k) $ and hence Proposition~\ref{prop_imp} implies that the coefficient of $e^{-\beta(\lambda_j, I_j)}$ in $L(\lambda_k, I_k)$ is positive.
    \end{enumerate}
\end{proof}
\begin{theorem}\label{mainthm_unspl}
    Let $r,s \geq 0$. Let $(\lambda_1,I_1)$, $(\lambda_2, I_2),\ldots, (\lambda_r, I_r)$ and $(\mu_1, J_1)$, $(\mu_2, J_2), $ $\ldots,$ $ (\mu_s,J_s) \in \cC$. We have
    \begin{equation}\label{eq_unspl_mainthm}
        \sum_{k=1} ^{r} L(\lambda_k, I_k) = \sum_{k=1} ^{s} L(\mu_k, J_k)
    \end{equation} if and only if
     $r=s$ and there exists a permutation $\sigma \in \mathfrak S _r$ such that $(\lambda_k, I_k) \approx (\mu_{\sigma(k)}, J_{\sigma(k)})$ for $1 \leq k \leq r$.
\end{theorem}
\begin{proof}
The reverse implication easily follows from part(2) of Proposition~\ref{prop_new}. We prove the forward implication by
 induction on $m:= \min\{r,s\}$. If $m=0$, then $r=s=0$ follows from (2) of Lemma~\ref{main_lemma}.
Suppose $m \geq 1$. Without loss of generality assume that $(\lambda_1, I_1)$ is maximal among $\{(\lambda_1, I_1)$ $, \cdots, (\lambda_r, I_r), (\mu_1, J_1), \cdots, (\mu_s, J_s) \}$ viewed as elements of $\cP / \approx$. By (2) of Lemma~\ref{main_lemma}, we see that the coefficient of $e^{-\beta(\lambda_1, I_1)}$ is non-zero in the left-hand side of the Eq.~\eqref{eq_unspl_mainthm}, and therefore also on the right-hand side. 
Now by (1) of Lemma~\ref{main_lemma} there exists $k$ such that $(\mu_k, J_k) \succeq (\lambda_1, I_1)$. But by maximality of $(\lambda_1, I_1)$ we conclude that $(\mu_k, J_k) \approx (\lambda_1, I_1)$. Therefore we may cancel $L(\lambda_1, I_1) = L(\mu_k, J_k)$ from both sides of Eq.~\eqref{eq_unspl_mainthm}, thereby reducing the value of $m$ by $1$.
\end{proof}
\begin{corollary}\label{corollary_unspl_ufpvm}
    Let $(\lambda_1,I_1)$, $(\lambda_2, I_2), \ldots, (\lambda_r, I_r)$ and $(\mu_1, J_1)$, $(\mu_2, J_2), \ldots, (\mu_r,J_r) \in \cP$ such that  all $I_k, J_k$ are connected (possibly empty). Then
    \begin{equation}\label{eq2_unspl_mainthm_cor}
        \prod_{k=1} ^{r} ch_{\mathfrak h} (M(\lambda_k, I_k)) = \prod_{k=1} ^{r} ch_{\mathfrak h} (M(\mu_k, J_k))
    \end{equation} if and only if      $\sum_{k=1} ^r \lambda_k = \sum_{k=1} ^r \mu_k$ and there exists
         $\sigma \in \mathfrak S _r$ such that $(\lambda_k, I_k) \approx (\mu_{\sigma(k)}, J_{\sigma(k)})$ for all $1\le k\le r$.
\end{corollary}
\begin{proof}
We will prove only the forward direction, the converse following easily from Proposition~\ref{prop_new}.
    Rewriting \eqref{eq2_unspl_mainthm_cor} using the character formula for parabolic Verma modules, we get
    \begin{equation}\label{eqn-wcf}
        \prod _{k=1} ^r e^{\lambda_k}\frac{U(\lambda_k, I_k)}{U(0, S)} = \prod _{k=1} ^r e^{\mu_k}\frac{U(\mu_k, J_k)}{U(0, S)}
    \end{equation}
    Comparing the highest weights on both sides of \eqref{eq2_unspl_mainthm_cor}, we get $\sum_{k = 1} ^r \lambda_k = \sum_{k=1} ^r \mu_k$. Therefore \eqref{eqn-wcf} gives: 
    \begin{equation*}
        \prod _{k=1} ^r U(\lambda_k, I_k) = \prod _{k=1} ^r U(\mu_k, J_k)
    \end{equation*}
    Note that $U(\lambda, I) = 1$ iff $I = \emptyset$. Ignoring such trivial terms in the above product on both sides we have up to a relabelling 
    \begin{equation*}
        \prod _{k=1} ^t U(\lambda_k, I_k) = \prod _{k=1} ^s U(\mu_k, J_k)
    \end{equation*}
    where, now $\{(\lambda_1, I_1), \ldots, (\lambda_t, I_t), (\mu_1, J_1), \ldots, (\mu_s, J_s)\} \subseteq \cC$.
Taking log on both sides, and applying Theorem~\ref{mainthm_unspl} we get $t=s$ and there is a permutation $\sigma \in \mathfrak S_t$ such that $(\lambda_k, I_k) \approx (\mu_{\sigma k}, J_{\sigma k})$. But $s=t$ implies that the number of trivial terms on both sides was also equal to begin with. Extending $\sigma$ trivially to a bijection of $\{1, \ldots, r\}$, we get the required permutation (because $(\lambda, \emptyset) \approx (\mu, \emptyset)$ for any $\lambda, \mu \in \mathfrak h ^*$).
\end{proof}
\noindent
In particular, Theorem~\ref{thm_intro_ufpvm} is now immediate from the above corollary.
\section{Unique Factorization for Restricted Parabolic Verma Modules}\label{section5}
 In this section, we prove the unique factorization of tensor products for certain classes of parabolic Verma modules restricted to compatible subalgebras of $\mathfrak g$.  

\subsection{} We begin with some auxiliary results. 
Recall that $\mathfrak g$ is a symmetrizable Kac-Moody algebra whose simple roots are $\{\alpha_1, \ldots, \alpha_n\}$. Fix an equivalence relation $\sim$ on $S=\{1, \ldots, n\}$. This gives rise to a set partition of $S$.
\begin{definition}
    Let $K \subseteq S$ be such that $K$ is a union of equivalence classes. We say that $\widehat{K}$ is a \emph{lift} of $K$ if 
    \begin{enumerate}
        \item $\widehat{K} \subseteq K$ 
        \item $\widehat{K}$ is connected
        \item $\widehat{K}$ meets every equivalence class in $K$.
    \end{enumerate}
\end{definition}
\begin{definition}\label{k hat}
    Let $K \subseteq S$ be such that $K$ is a union of equivalence classes. We say that $K$ is \emph{equiconnected} if there exists a lift $\widehat{K}$ of $K$ such that given any lift $\bar{K}$ of $K$ and any equivalence class $E$, $|\bar{K} \cap E| \geq |\widehat{K} \cap E|$. Any such lift $\widehat{K}$ will be referred to as a \emph{lean lift}.  
\end{definition}
\begin{remark}\begin{upshape} \label{rmk_minlift_card_unique}
    Note that if $\widehat{K}$ and $\widehat{K}'$ are two lean lifts of $K$, then given any equivalence class $E$ we have \begin{equation} \label{eqn_unique_minimal_lifts}    |\widehat{K} \cap E| = |\widehat{K}' \cap E|
\end{equation}   
\end{upshape}\end{remark}
    
    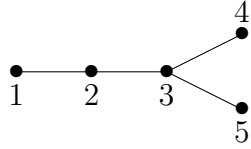
\begin{figure}[h]
        \centering
        \begin{tikzpicture}
            \node at (0,0) {$\bullet$};
            \node at (1,0) {$\bullet$};
            \node at (2,0) {$\bullet$};
            \node at (3,0.5) {$\bullet$};
            \node at (3,-0.5) {$\bullet$};
            \node at (0,-0.3) {$1$};
            \node at (1,-0.3) {$2$};
            \node at (2,-0.3) {$3$};
            \node at (3,0.8) {$4$};
            \node at (3,-0.8) {$5$};
            \draw (0,0) -- (2,0) -- (3,0.5);
            \draw (2,0) -- (3,-0.5);            
        \end{tikzpicture}\caption{The equivalence classes are given by $\{1\}$, $\{2\}$, $\{3\}$ and $\{4,5\}$.}\label{exmpl_non-uni-min_lift}
    \end{figure}
    In figure~\ref{exmpl_non-uni-min_lift}, there are two lean lifts for $S = \{1,2,3,4,5\}$ namely: $\{1,2,3,4\}$ and $\{1,2,3,5\}$.

	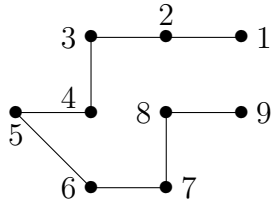
\begin{figure}[h]
		\centering
		\begin{tikzpicture}
			\node at (0,0) {$\bullet$};
			\node at (0,-.3) {$5$};
			\node at (1,0) {$\bullet$};
			\node at (0.7,0.2) {$4$};
			\node at (1,1) {$\bullet$};
			\node at (0.7,1) {$3$};
			\node at (1,-1) {$\bullet$};
			\node at (0.7,-1) {$6$};
			\node at (2,1) {$\bullet$};
			\node at (2,1.3) {$2$};
			\node at (3,1) {$\bullet$};
			\node at (3.3, 1) {$1$};
			\node at (2,0) {$\bullet$};
			\node at (1.7,0) {$8$};
			\node at (3,0) {$\bullet$};
			\node at (3.3,0) {$9$};
			\node at (2,-1) {$\bullet$};
			\node at (2.3,-1) {$7$};
			\draw (3,0) -- (2,0) -- (2,-1) -- (1,-1) -- (0,0) -- (1,0) -- (1,1) -- (2,1) -- (3,1);
		\end{tikzpicture}
		\caption{The equivalence classes are given by $\{5\}, \{3,4,6\}$, $\{2,7,8\}$ and $\{1,9\}$.}\label{exmpl_not-tilde-connected}
	\end{figure}
 In figure~\ref{exmpl_not-tilde-connected}, for $K = \{1, \ldots, 9\}$, the subsets $K_1 = \{1,2,3,4,5\}$ and $K_2 = \{5,6,7,8,9\}$ are both lifts but $|K_1 \cap \{3,4,6\}| > |K_2 \cap \{3,4,6\}|$ while $|K_1 \cap \{2,8,7\}| < |K_2 \cap \{2,8,7\}|$. It follows that $K$ is not equiconnected.
\begin{definition}
    We call $\lambda \in \mathfrak h ^*$ to be \emph{symmetric} if $\forall \; i, j \in S$ such that $i \sim j$, we have $\lambda(\alpha_i ^\vee) = \lambda(\alpha _j ^{\vee})$.
\end{definition}
\begin{remark}\begin{upshape}
    Suppose $\widehat{K}$ is a lean lift and $\bar{K}$ is some lift of $K$. It is immediate from Remark~\ref{rmk_minlift_card_unique} that if $\lambda$ is symmetric,  then $\beta(\lambda, \bar{K}) - \beta(\lambda, \widehat{K}) \in Q^+$. Further $\beta(\lambda, \bar{K}) = \beta(\lambda, \widehat{K})$ if and only if $\bar{K}$ is a lean lift of $K$.
\end{upshape}\end{remark}
Define $\Bar{\cC}$ to be the set of pairs $(\lambda, I) \in \cC$ satisfying
\begin{enumerate}
    \item $\lambda$ is symmetric
    \item $I$ is a union of equivalence classes of $\sim$
    \item $I$ is equiconnected.
\end{enumerate}

Define $\overline{Q} := \oplus _{[j] \in S/ \!\sim} \; \mathbb{Z}\gamma _{[j]}$. The set of all non-negative integer linear combinations of $\{\gamma_{[j]} \; : \; [j] \in S/\!\! \sim\}$ is denoted by $\overline{Q} ^+$. Define the map $$\pi: Q \rightarrow \overline{Q}$$ where $\alpha_i$ maps to (the formal symbol) $\gamma_{[i]}$. This induces a map from $\mathbb Z [[\{e^{-\alpha_i} | i\in S\}]]$ to $\mathbb Z [[\{e^{-\gamma_{[i]}}| [i] \in S/\!\!\sim \}]]$ which we again denote by $\pi$.
For $(\lambda, K) \in \Bar{\cC}$, we define $$\bar{\beta}(\lambda, K):= \pi(\beta(\lambda, \widehat{K}))$$ for any lean lift $\widehat{K}$ of $K$ as in Definition~\ref{k hat}. Observe that this does not depend on the choice of $\widehat{K}$ by \eqref{eqn_unique_minimal_lifts}.
\begin{lemma}
    If $(\lambda, I), (\mu, J) \in \bar{\mathcal C}$ are such that $\bar{\beta}(\lambda, I) = \bar{\beta}(\mu, J)$, then $(\lambda, I) \approx (\mu, J)$.
\end{lemma}
\begin{proof}
    Since $\bar{\beta}(\lambda, I) = \bar{\beta}(\mu, J)$, we have for any choice of $\hat{I}$ and $\hat{J}$:
    $$\pi\left(\sum_{i \in \hat{I}} (\lambda+\rho)(\alpha_i ^\vee)\alpha_i \right) = \pi\left(\sum_{i \in \hat{J}} (\mu+\rho)(\alpha_i ^\vee)\alpha_i\right )$$
    $$\sum_{[i] \in I/\sim}\left(\sum_{j \in [i] \cap \hat{I}} (\lambda+\rho)(\alpha_i ^\vee)\right) \gamma _{[i]} = \sum_{[i] \in J/\sim}\left(\sum_{j \in [i] \cap \hat{J}}(\mu+\rho)(\alpha_i ^\vee)\right)\gamma _{[i]}$$
    Since $\lambda + \rho$ and $\mu + \rho$ are regular dominant, $\gamma_{[i]}$'s are linearly independent and $I, J$ are unions of equivalence classes of $\sim$, it follows that $I = J$. Now, $$\sum_{[i] \in I/\sim}|[i] \cap \hat{I}|\cdot (\lambda+\rho)(\alpha_i ^\vee) \gamma _{[i]} = \sum_{[i] \in J/\sim}|[i] \cap \hat{J}| \cdot (\mu+\rho)(\alpha_i ^\vee)\gamma _{[i]}$$ 
    Since, $I$ and $J$ were equiconnected we have, for all $[i]$, $$|[i] \cap \hat{I}| = |[i] \cap \hat{J}|$$
    By comparing coefficients of $\gamma_{[i]}$ one sees that $(\lambda, I) \approx (\mu, J)$.
\end{proof}
\begin{remark}\begin{upshape}\label{rmk_partialorder-specialized}
        The relation $\succeq$ on $\cP$ can be restricted to $\Bar{\cC}$. It is elementary to check that for $(\lambda, I), (\mu, J) \in \Bar{\cC}$ we have $(\lambda, I) \succeq (\mu, J)$ if and only if one of the following holds:
        \begin{itemize}
            \item $I \supsetneq J$ or
            \item $I=J$ and $\bar{\beta}(\mu, J) - \bar{\beta}(\lambda, I) \in \overline{Q}^+$ 
        \end{itemize}
\end{upshape}\end{remark}
For $(\lambda, I) \in \cP$ let us write $\pi(L(\lambda, I)) = \sum \limits_{\gamma \in \overline{Q}} d ^{\lambda , I} _{\gamma} e^{-\gamma}$.
Observe that $d^{\lambda, I} _{\gamma}= \sum \limits_{\alpha} c^{\lambda, I} _{\alpha}$, where the sum runs over all $\alpha \in Q$ for which $\pi(\alpha) = \gamma$. The following is the analogue of Propositions~\ref{prop_new} and \ref{prop_imp}:
\begin{proposition} \label{prop_new_dup-4,5}
Suppose $(\lambda, I) \in \Bar{\cC}$.
\begin{enumerate}
    \item If $\gamma \in \overline{Q}$ is such that $\supp(\gamma) = I/\!\! \sim$ and $d^{\lambda, I} _{\gamma} \neq 0$ then $\gamma - \bar{\beta}(\lambda, I) \in \overline{Q}^+$.
    \item The coefficient of $e^{-\bar{\beta}(\lambda, I)}$ in $\pi(L(\lambda, I))$ is positive (i.e., $d^{\lambda, I} _{\bar{\beta}(\lambda, I)} > 0$) and independent of $\lambda$ for fixed $I$.
\end{enumerate}
\end{proposition}
\begin{proof}
Suppose $\alpha \in Q$ is such that $c^{\lambda, I} _{\alpha} \neq 0$ and $\pi(\alpha) = \gamma$. This implies that $\supp(\alpha)$ is a lift of $I$ (see part (5) in Proposition~\ref{prop_new}). If $\hat{I}$ is any lean lift of $I$ then  equiconnectedness of $I$ implies that $\pi(\beta(\lambda, \supp(\alpha))) - \pi(\beta(\lambda, \hat{I})) \in \overline{Q}^+$. Also, by (4) and (6) of Proposition~\ref{prop_new}, we have $\alpha - \beta(\lambda, \supp(\alpha)) \in Q^+$. 
Therefore by applying $\pi$ and combining with the previous observation we see that $\gamma - \bar{\beta}(\lambda, I) \in \overline{Q}^+$. This proves $(1)$.

Suppose if $\alpha$ is such that $c^{\lambda, I} _{\alpha} \neq 0$ and $\pi(\alpha) = \bar{\beta}(\lambda, I)$.  It follows from equiconnectedness of $I$ and part (5) in Proposition~\ref{prop_new} that $\alpha = \beta(\lambda, \hat{I})$ for some lean lift $\hat{I}$ of $I$. 
Therefore $$d^{\lambda, I} _{\bar{\beta(\lambda, I)}} = \sum_{\hat{I}} c^{\lambda, I} _{{\beta}(\lambda, \hat{I})}$$ where the sum runs over all lean lifts $\hat{I}$ of $I$ as in Definition~\ref{k hat}. Part (2) of the lemma now follows from part (4) of Proposition~\ref{prop_new} and Proposition~\ref{prop_imp}.
\end{proof}
\begin{lemma}\label{main_lemma-dup}
    Let $(\lambda_1, I_1), \ldots, (\lambda_r, I_r)$ be (not necessarily distinct) elements of $\bar{\cC}$ and let $L:= \sum_{i=1} ^r L(\lambda_k, I_k)$. 
    \begin{enumerate}
    \item For $(\mu, J) \in \bar{\cC}$, if the coefficient of $e^{-\bar{\beta}(\mu, J)}$ in $\pi(L)$ is non-zero, then $(\lambda_k, I_k) \succeq (\mu, J)$ for some $1 \leq k \leq r$.
    \item If $(\lambda_j, I_j)$ is maximal among $(\lambda_1, I_1), \ldots, (\lambda_r, I_r)$ then the coefficient of $e^{-\bar{\beta}(\lambda_j, I_j)}$ in $\pi(L)$ is positive. In particular, $\pi(L) \neq 0$ if $r > 0$.
    \end{enumerate}
\end{lemma}
\begin{proof}
Suppose that the coefficient of $e^{-\bar{\beta}(\mu, J)}$ is non-zero, then this monomial must come from $\pi(L(\lambda_k, I_k))$ for some $I_k \supseteq J$. If this containment is proper then we are done. Suppose now that $J=I_k$, then by part (1) of Proposition~\ref{prop_new_dup-4,5} and Remark~\ref{rmk_partialorder-specialized}, it follows that  $(\lambda _k, I_k ) \succeq (\mu, J)$.
By maximality of $(\lambda_j, I_j)$, part (2) of the lemma follows from part (1) together with Proposition~\ref{prop_new_dup-4,5} and Remark~\ref{rmk_partialorder-specialized}.
\end{proof}

\subsection{}
We are now ready to state and prove our main theorem for the restricted normalized Weyl numerators. 
\begin{theorem}\label{new_theorem_specialized}
    Let $\{(\lambda_k, I_k)\}_{k = 1} ^{r}$ and $\{(\mu_k, J_k)\}_{k=1} ^{s}$ be subsets of $\Bar{\cC}$. Then the following are equivalent:
    \begin{enumerate}
        \item  $        \sum_{k=1} ^{r} L(\lambda_k, I_k) = \sum_{k=1} ^{s} L(\mu_k, J_k)$ 
        \medskip
        \item        $ \sum_{k=1} ^{r} \pi(L(\lambda_k, I_k)) = \sum_{k=1} ^{s} \pi(L(\mu_k, J_k))$
        \medskip
    \item  $r=s$ and there exists $\sigma \in \mathfrak S _r$ such that $(\lambda_k, I_k) \approx (\mu_{\sigma(k)}, J_{\sigma(k)})$ for all $k$.
    \end{enumerate}
\end{theorem}
\begin{proof}
The statement (2) follows from (1) by applying $\pi$. The statement (1) follows from (3) by Proposition~\ref{prop_new}. We now prove that statement (2) implies (3).
We proceed by induction on $m:= \min\{r,s\}$. If $m=0$, then $r=s=0$ follows from part (2) of Lemma~\ref{main_lemma-dup}.
Suppose $m \geq 1$. Without loss of generality we assume that $(\lambda_1, I_1)$ is maximal among   $\{(\lambda_1, I_1),$ $\cdots,$ $(\lambda_r, I_r),$ $(\mu_1, J_1), \ldots, (\mu_s, J_s) \}$. By part (2) of Lemma~\ref{main_lemma-dup}, we see that the coefficient of $e^{-\bar{\beta}(\lambda_1, I_1)}$ is non-zero in the LHS of ~\eqref{eq_unspl_mainthm}. 
Now by part (1) of Lemma~\ref{main_lemma-dup} there exists $k$ such that $(\mu_k, J_k) \succeq (\lambda_1, I_1)$. But by maximality of $(\lambda_1, I_1)$ we conclude that $(\mu_k, J_k) \approx (\lambda_1, I_1)$. Therefore we may cancel $\pi(L(\lambda_1, I_1)) = \pi(L(\mu_k, J_k))$ from both sides of ~\eqref{eq_unspl_mainthm}, thereby reducing the value of $m$ by $1$.
\end{proof}

\subsection{} Now start with a \emph{Dynkin diagram} $G$. Let $\mathfrak g$ be the associated KMA. Recall that $S = \{1,\ldots, n\}$ is the vertex set of $G$ (or equivalently the set of simple roots). Given an equivalence relation $\sim$ on $S$ we define the subspace $\kk$ of $\mathfrak h$ as follows:
\begin{equation}\label{eqn_def_k}
\kk := \cap_{i \sim j} Ker(\alpha_i - \alpha_j)
\end{equation}
Then on $\kk$ we have $\alpha_i = \alpha_j$ if $i \sim j$ (actually, $\mathfrak k$ is the largest subspace of $\mathfrak h$ where we have $\alpha_i = \alpha_j$ whenever $i \sim j$). Moreover, for a given $i \in S$ the element $\omega_{[i]} ^\vee:= \sum_{k \sim i} \omega_k ^{\vee} \in \mathfrak k$ satisfies $\alpha_j(\omega_{[i]} ^\vee) = \delta_{[i],[j]}$. Here $\omega_k ^{\vee}$ is a fixed choice of fundamental co-weight associated to the simple root $\alpha_k$. i.e., $\alpha_l (\omega_k ^{\vee}) = \delta_{k,l} \; \forall l$. Therefore we have $\alpha_i = \alpha_j$ on $\kk$ iff $i \sim j$. Moreover, we conclude the following:
\begin{proposition}
    Any collection of simple roots corresponding to distinct $\sim$ orbit representatives forms a linearly independent set when restricted to $\kk$.  \qed
\end{proposition}

\noindent
Let $\mathfrak s$ be any subspace of $\mathfrak h$ such that 
\begin{enumerate}
    \item $ \alpha_i|_{\mathfrak s} =  \alpha_j|_{\mathfrak s}$ whenever $i \sim j$.
    \item $\{\alpha_i|_{\mathfrak s} \; : \; i \in S/ \!\!\sim\}$ is a linearly independent subset of $\mathfrak s ^*$.
\end{enumerate}
Note that $\mathfrak k$ is one such subspace of $\mathfrak h$.
Denote the restriction map by $\p: \mathfrak h ^* \rightarrow \mathfrak s ^*$. This map extends uniquely to a map (which we again denote by $\p$) from $ \mathbb{C}[[\{e^{-\alpha_i}: i\in S\}]] $ to $ \mathbb{C}[[\{e^{-\p(\alpha_i)}: [i]\in S/ \!\!\sim\}]]$. Observe that the map $\p$ is the same as the map $\pi$ when one identifies $\gamma_{[i]}$ with $\p(\alpha_i)$. Therefore we have,
\begin{corollary} \label{cor_tensor_pvm}
    Let $(\lambda_1,I_1)$, $(\lambda_2, I_2), \ldots, (\lambda_r, I_r)$ and $(\mu_1, J_1)$, $(\mu_2, J_2), \ldots, (\mu_r,J_r) \in \bar{\cC}$ except that we now allow the $I_k, J_k$ to be empty.
    Then 
    \begin{equation}\label{eq2_spl_mainthm}
        \p(ch_{\mathfrak h} (\bigotimes_{k=1} ^{r} M(\lambda_k, I_k))) = \p(ch_{\mathfrak h} (\bigotimes_{k=1} ^{r} M(\mu_k, J_k)))
    \end{equation}
    if and only if
    \begin{enumerate}
        \item $\p(\sum_{k=1} ^r \lambda_k) = \p(\sum_{k=1} ^r \mu_k$)
        \item $\exists \: \sigma \in \mathfrak S _r$ such that $(\lambda_k, I_k) \approx (\mu_{\sigma(k)}, J_{\sigma(k)})$ for all $k$.
    \end{enumerate}
\end{corollary}
\begin{proof}
    Rewriting \eqref{eq2_spl_mainthm} using the character formula for parabolic Verma modules, we get
    \begin{equation*}
        \prod _{k=1} ^r \p(e^{\lambda_k})\frac{\p(U(\lambda_k, I_k))}{\p(U(0, S))} = \prod _{k=1} ^r \p(e^{\mu_k})\frac{\p(U(\mu_k, J_k))}{\p(U(0, S))}
    \end{equation*}
    Comparing the highest weights on both sides of \eqref{eq2_spl_mainthm}, we get $\p(\sum_{k = 1} ^s \lambda_k) = \p(\sum_{k=1} ^s \mu_k)$. Therefore, 
    \begin{equation*}
        \prod _{k=1} ^r \p(U(\lambda_k, I_k)) = \prod _{k=1} ^r \p(U(\mu_k, J_k))
    \end{equation*}
    We now proceed as in the proof of Corollary~\ref{corollary_unspl_ufpvm}.
    Note that $\p(U(\lambda, I)) = 1$ iff $I = \emptyset$. Ignoring these trivial terms in the above product and relabelling 
    \begin{equation*}
        \prod _{k=1} ^s \p(U(\lambda_k, I_k)) = \prod _{k=1} ^t \p(U(\mu_k, J_k))
    \end{equation*}
    where, now $\{(\lambda_1, I_1), \ldots, (\lambda_s, I_s), (\mu_1, J_1), \ldots, (\mu_t, J_t)\} \subseteq \bar{\cC}$.
Taking logarithm on both sides of the above equation, and applying Theorem~\ref{new_theorem_specialized} we get $s=t$ and there is a permutation $\sigma \in \mathfrak S_t$ such that $(\lambda_k, I_k) \approx (\mu_{\sigma k}, J_{\sigma k})$. The rest of the argument is exactly as in Corollary~\ref{corollary_unspl_ufpvm}.

For the converse part, the second condition implies that $U(\lambda_k, I_k) = U(\mu_{\sigma(k)}, J_{\sigma(k)})$. Therefore we have
    \begin{equation*}
        \prod _{k=1} ^r \frac{U(\lambda_k, I_k)}{U(0, S)} = \prod _{k=1} ^r\frac{U(\mu_k, J_k)}{U(0, S)}
    \end{equation*}
    Now by applying $p$ on the both sides of above equation
    and multiplying $e^{\p(\sum_{k=1}^r \lambda_k)}$ on the left hand side and multiplying  $e^{\p(\sum_{k=1}^r \mu_k)}$ on the right hand side of the equation gives us 
    \begin{equation*}
        \prod _{k=1} ^r \p(e^{\lambda_k})\frac{\p(U(\lambda_k, I_k))}{\p(U(0, S))} = \prod _{k=1} ^r \p(e^{\mu_k})\frac{\p(U(\mu_k, J_k))}{\p(U(0, S))}.
    \end{equation*}
    Now  using the character formula, we conclude the result. 
\end{proof}

\section{Unique Factorization Of Restricted Parabolic Vermas}\label{section6}
In this section, we will apply the results of the previous section to the special case of fixed point subalgebras of Dynkin diagram automorphisms and twisted graph automorphisms when $\mathfrak g$ is of untwisted affine type. 
\subsection{Graph automorphisms}
\begin{proposition}\label{prop_graph_tildeconnected}
    Let $G = (V(G), E(G))$ be a connected graph and $\Gamma$ be a subgroup of the group of all automorphisms of $G$. Then there exists a connected subgraph of $G$ whose vertex set intersects every $\Gamma$ orbit in $G$ at exactly one point.
\end{proposition}
\begin{proof}
    Let $\mathcal{A}$ denote the set of all subsets of $V(G)$ that intersect any $\Gamma$ orbit in $G$ in at most one point and whose induced subgraph is connected. Clearly $\mathcal{A}$ is non-empty because it contains all the singleton subsets of $V(G)$.

\medskip
    Let $M \in \mathcal{A}$ be a maximal element with respect to the containment partial order. For any graph automorphism $\omega \in \Gamma$ we see that, $\omega (M)$ also belongs to $\mathcal{A}$. Suppose that $M$ does not intersect a $\Gamma$-orbit in $G$. This means that
    $$N := \bigcup\limits_{\omega \in \Gamma} \omega (M) \neq V(G).$$ But since $G$ is connected, there exist elements $x \in N$ and $y \in V(G) - N$ such that $(x, y) \in E(G)$. But there exists some $\omega \in \Gamma$ for which $\omega(x) \in M$. Therefore we would have $M \cup \{\omega (y)\} \in \mathcal{A}$ which is a contradiction to the assumption that $M$ was maximal. So, $M$ intersects every $\Gamma$-orbit in $G$. 
\end{proof}

Any graph automorphism $\omega$ of $G$ induces a Lie algebra automorphism of $\mathfrak g$ (which will be referred to as \emph{diagram automorphisms}) described as follows: It maps the generators $e_i, h_i$ and $f_i$ to $e_{\omega i}, h_{\omega i}$ and $f_{\omega i}$ respectively for all $i \in S$. This assignment extends uniquely to a Lie algebra automorphism of the derived subalgebra of $\mathfrak g$. This map can be extended to an automorphism of $\mathfrak g$ in a unique way if we impose the condition that it preserves the standard invariant bilinear form and has order same as that of $\omega$. Such an automorphism preserves $\mathfrak h$ and its induced action on $\mathfrak h^*$ permutes the simple roots. See \cite[\S3.2]{Fuchs_1996}.

\medskip
Let $\Gamma$ be a subgroup of diagram automorphisms of $\mathfrak g$. 
Denote by $\mathfrak g ^{\Gamma}$ (resp. $\mathfrak h ^{\Gamma}$) the fixed point subalgebra of $\mathfrak g$ (resp. $\mathfrak h$) with respect to $\Gamma$. 

\begin{proposition}\label{prop-h_gamma=k}
    Let $A$ be a GCM whose nullity is at most $1$. Then for $\mathfrak g = \mathfrak g (A)$ we have,
    \begin{equation*}
        \mathfrak h^{\Gamma} = \bigcap _{w \in \Gamma, \; i \in S} Ker(\alpha_i - \alpha_{w(i)})
    \end{equation*}
    i.e., $\mathfrak h ^{\Gamma} = \mathfrak k$ as in the notation of \eqref{eqn_def_k}.
\end{proposition}
\begin{proof}
    It is easy to see that $\mathfrak h ^{\Gamma}$ is a subset of $\mathfrak k$, since for all $h \in \mathfrak h$ and $i \in S$ we have $$\alpha_i (h) = \alpha_{\omega(i)} (\omega (h))$$ for any graph automorphism $\omega$.
    Now one checks that the dimension of $\mathfrak h ^{\Gamma}$ is either the number of orbits of $\Gamma$'s action on the Dynkin diagram or one more to it depending on the nullity of A being $0$ or $1$ (see the construction in \cite[\S3.2]{Fuchs_1996}. Basically, in the nullity $1$ case one can find a $d \in \mathfrak h \backslash \mathrm{span}\{\alpha_i ^{\vee} \;:\; i \in S\}$ such that $\omega(d) = d$). In both cases, it matches the dimension of $\mathfrak k$.
\end{proof}
\begin{example}
It can be checked that for the following graph, the conclusion of the above proposition is not true. Note that the nullity of the GCM associated to this graph is $2$.
    
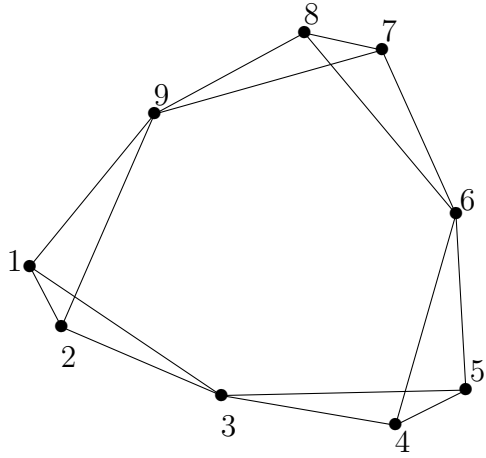
\begin{figure}[h]
    \begin{tikzpicture}[scale=0.6]
\draw (-9,0)-- (-8.3,-1.34);
\node at (-9,0) {$\bullet$};
\draw (-8.3,-1.34)-- (-4.76,-2.86);
\node at (-8.3, -1.34) {$\bullet$};
\draw (-4.76,-2.86)-- (-0.9,-3.52);
\node at (-4.76, -2.86) {$\bullet$};
\draw (-0.9,-3.52)-- (0.66,-2.74);
\node at (-0.9, -3.52) {$\bullet$};
\draw (0.66,-2.74)-- (0.44,1.16);
\node at (0.66, -2.74) {$\bullet$};
\draw (0.44,1.16)-- (-1.2,4.8);
\node at (0.44, 1.16) {$\bullet$};
\draw (-1.2,4.8)-- (-2.92,5.18);
\node at (-1.2, 4.8) {$\bullet$};
\draw (-2.92,5.18)-- (-6.24,3.38);
\node at (-2.92,5.18) {$\bullet$};
\draw (-6.24,3.38)-- (-9,0);
\node at (-6.24,3.38) {$\bullet$};
\draw (-8.3,-1.34)-- (-6.24,3.38);
\draw (-9,0)-- (-4.76,-2.86);
\draw (-4.76,-2.86)-- (0.66,-2.74);
\draw (-0.9,-3.52)-- (0.44,1.16);
\draw (0.44,1.16)-- (-2.92,5.18);
\draw (-6.24,3.38)-- (-1.2,4.8);
\node at (-9.34,0.13) {$1$};
\node at (-8.14,-2.01) {$2$};
\node at (-4.6, -3.53) {$3$};
\node at (-0.74,-3.89) {$4$};
\node at (0.92,-2.31) {$5$};
\node at (0.7,1.49) {$6$};
\node at (-1.04,5.23) {$7$};
\node at (-2.76,5.61) {$8$};
\node at (-6.08,3.81) {$9$};
\end{tikzpicture}
\end{figure}

Define an equivalence relation on the set of nodes $S$ of the Dynkin diagram of $\mathfrak g$ as follows:
$$i \sim j \iff \; \exists \; \omega \in \Gamma= \text{Aut}(G) \text{ such that } \omega(i) = j$$
\end{example}

In the view of Proposition~\ref{prop_graph_tildeconnected}, for the equivalence relation $\sim$ induced by $\Gamma$, any connected subset of $S$ which is a union of equivalence classes is indeed equiconnected. Also, it is elementary to check that $\lambda \in \mathfrak h^*$ is symmetric if and only if $\lambda(\omega(h)) = \lambda(h)$ for all $\omega \in \Gamma$ and $h \in \mathfrak h$.
\begin{corollary}\label{cor_fixedpointsubalg}
    Let $A$ be a symmetrizable Generalised Cartan Matrix whose nullity is at most $1$. Suppose $(\lambda_1, I_1), \ldots, (\lambda_r, I_r), (\mu_1, J_1), \ldots, (\mu_r, J_r) \in \bar{\cC}$ with the exception that the $I_k$ and $J_k$ could be empty. Then
    \begin{equation*}
    \prod _{k=1} ^r ch_{\mathfrak h ^{\Gamma}} M(\lambda_k, I_k) = \prod _{k=1} ^r ch_{\mathfrak h ^{\Gamma}} M(\mu_k, J_k)
    \end{equation*} if and only if
    \begin{enumerate}
        \item $\sum_{k=1} ^r \lambda_k = \sum_{k=1} ^r \mu_k$
        \item there exists $\sigma \in \mathfrak S_r$ such that $(\lambda_k, I_k) \approx (\mu_{\sigma(k)}, J_{\sigma(k)})$.
    \end{enumerate}
\end{corollary}
\begin{proof}
    Part (2) follows from Proposition~\ref{prop-h_gamma=k} and Corollary~\ref{cor_tensor_pvm}. We also have by the same that $\sum_{k=1} ^r \lambda_k = \sum_{k=1} ^r \mu_k$ when restricted to $\mathfrak h ^{\Gamma}$. But since the $\lambda_k$'s and $\mu_k$'s are $\Gamma$-invariant (in other words symmetric) we have $\sum_{k=1} ^r \lambda_k = \sum_{k=1} ^r \mu_k$ on the whole of $\mathfrak h$. The converse part follows from the character formula ~\eqref{parabolicverma}.
\end{proof}
In particular, Theorem~\ref{mainthm2intro} now follows from the above corollary.
\subsection{Twisted graph automorphisms}
   Let $\mathfrak g$ be an untwisted affine Lie algebra. Then $\mathfrak g$ can be realized very explicitly 
as follows: $$\mathfrak g = \mathfrak g_0\otimes \mathbb{C}[t, t^{-1}]\oplus \mathbb{C}c \oplus \mathbb{C}d,$$
   where $\mathfrak g_0$ is the underlying finite-dimensional simple Lie algebra, $c$ is the central element and $d$ is the derivation. 
   Let $\sigma$ be a diagram automorphism of the underlying finite-dimensional simple algebra $\mathfrak g _0$. This induces an automorphism $\tau$ of $\mathfrak g$ called the \emph{twisted diagram automorphism} described as follows:
   $$\tau(x \otimes t^{k}) := e^{-2 k\pi i/n} \sigma(x) \otimes t^k, x\in \mathfrak g_0, \quad \quad \tau(c) := c \quad \quad \tau(d) := d$$where $n$ is the order of $\sigma$.
 But $\sigma$ also induces a diagram automorphism $\bar{\tau}$ of $\mathfrak g$ which is obtained by fixing the affine node. That is,$$\bar{\tau}(x \otimes t^k) := \sigma(x) \otimes t^k \quad \quad  \bar{\tau} (c) := c \quad \quad \bar{\tau}(d) := d,$$
  see \cite[\S9.5, \S18.3 and \S18.4]{Carter} for more details. 

  \medskip
It is well known that we can obtain the twisted affine Lie algebras from the untwisted ones as fixed point subalgebras of twisted diagram automorphisms (see \cite[\S 18.4]{Carter}). We thus have the following corollary concerning the tensor products of parabolic Verma modules of an untwisted affine Lie algebra restricted to the corresponding twisted affine algebra (obtained as a fixed point subalgebra): 
\begin{corollary}\label{cor_twisted}
    Let $\mathfrak g$ be an untwisted affine Lie algebra. Let $\tau$ be a twisted diagram automorphism of $\mathfrak g$ (and $\bar{\tau}$ be the associated diagram automorphism). Suppose $(\lambda_1, I_1), \ldots, (\lambda_r, I_r),$ $(\mu_1, J_1), \ldots,$ $(\mu_r, J_r) \in \bar{\cC}$ (with respect to $\bar{\tau}$) with the exception that the $I_k$ and $J_k$ could be empty. We have
    \begin{equation*}
    \prod _{k=1} ^r ch_{\mathfrak h ^{\tau}} M(\lambda_k, I_k) = \prod _{k=1} ^r ch_{\mathfrak h ^{\tau}} M(\mu_k, J_k)
    \end{equation*}
    if and only if
    \begin{enumerate}
        \item $\sum_{k=1} ^r \lambda_k = \sum_{k=1} ^r \mu_k$
        \item there exists $\sigma \in \mathfrak S_r$ such that $(\lambda_k, I_k) \approx (\mu_{\sigma(k)}, J_{\sigma(k)}) $.
    \end{enumerate} 
\end{corollary}
\begin{proof}
    The proof is immediate from Corollary~\ref{cor_fixedpointsubalg}, because $\tau = \bar{\tau}$ when restricted to $\mathfrak h$. 
\end{proof}

\bibliographystyle{plain}
\bibliography{biblio}

\end{document}